\documentclass[10pt, a4paper, reqno, english]{amsart}
\usepackage[utf8]{inputenc}
\usepackage{amsmath,amsthm,amssymb}
\usepackage{tikz}
\usepackage{hyperref}

\def\SW{\mathrm{SW}}
\newcommand\ddd{\,\mathrm{d}}

\theoremstyle{plain}
\newtheorem{prop}{Proposition}
\newtheorem{theorem}[prop]{Theorem}
\theoremstyle{remark}
\newtheorem*{remark*}{Remark}

\newcommand{\Mod}[1]{\ (\mathrm{mod}\ #1)}

\title[The inverse problem for the Steiner--Wiener index]{The inverse problem for the Steiner--Wiener index via additive number theory}

\author{Christian Bernert}
\address{Institute for Science and Technology Austria\\
Am Campus 1\\
3400 Klosterneuburg\\
Austria}
\email{christian.bernert@ist.ac.at}

\author{Joshua Shaw}
\address{Preston\\United Kingdom}
\email{shawjoshua265@gmail.com}

\date{\today}

\begin{document}

\begin{abstract}
We show that, for any given $k \ge 2$,  every sufficiently large number appears as the Steiner--Wiener $k$ index of a graph.
    
\end{abstract}

\maketitle

\section{Introduction}

Recall that for a connected graph $G$ and a subset $S \subset V(G)$ of its vertices, the \emph{Steiner distance} $d(S)$ is defined as the smallest number of edges in a subgraph of $G$ whose vertex set contains $S$.
Note that for $|S|=2$, the Steiner distance $d(S)$ is just the ordinary distance in $G$ between the two vertices in $S$.

For a positive integer $k \ge 2$, the \emph{Steiner--Wiener $k$ index} of $G$ was defined in \cite{lmg16} (but essentially already discovered earlier in \cite{dankelmann96}) as
\[\SW_k(G)=\sum_{S \subset V(G), |S|=k} d(S).\]
Note that $\SW_2(G)$ is more classically known as the \emph{Wiener index} of $G$.

Motivated by questions in chemical graph theory, a lot of research has been devoted to studying $\SW_k(G)$ for special classes of graphs, see \cite{maofurtula21} for a survey.

The \emph{inverse problem} asks for the possible values of $\SW_k(G)$ for a fixed $k$, where $G$ ranges through all graphs.
In \cite{gyc94}, it was shown that for $k=2$, all positive integers except for $2$ and $5$ appear as the Wiener index of a graph.
In \cite{zgwj19}, this was extended to the cases $k \in \{3,4,5\}$, where it was shown that all but finitely many values appear as $\SW_k(G)$ for some $G$. In fact, for $k=3$ by computer search an explicit list of the $34$ exceptions was given.

Our main result resolves the inverse problem for all $k$.

\begin{theorem}\label{thm:main}
    Let $k \ge 2$ be a positive integer. Then for all but finitely many positive integers $n$, there is a graph $G$ with $\SW_k(G)=n$.
\end{theorem}

Note that our method does not give particularly good bounds on the size of the exceptional set, it seems that different techniques are required to improve on this.

We also remark that from the point of view of applications in chemistry, it is more natural to restrict to special classes of graphs, such as trees. The inverse problem then becomes much harder. It was resolved for $k=2$ independently in \cite{wagner06} and \cite{wangyu06}  but remains open for any $k \ge 3$.

We will deduce Theorem~\ref{thm:main} from the following two results:

\begin{prop} \label{prop:graph}
    For every sequence $a_1<\dots<a_r < n-1$ there is a graph $G$ with
    \[\SW_k(G)=\SW_k(S_n)-\binom{a_1}{k-1}-\dots-\binom{a_r}{k-1}.\]
\end{prop}

Here $S_n$ is the star graph on $n$ vertices and $n-1$ edges. Note that $\SW_k(S_n)=k \cdot \binom{n}{k}-\binom{n-1}{k-1}=(n-1)\binom{n-1}{k-1}$.

Star graphs were already used in \cite{zgwj19}. Our new observation is that one can construct \emph{nested} stars to allow  for more flexibility in the resulting values of $\SW_k$.

This naturally shifts our attention to the number-theoretic question of which values can be obtained as such a sum of binomial coefficients.

\begin{prop} \label{prop:ant}
    For every positive integer $d$ there are positive integers $s=s(d)$ and $m_0=m_0(d)$ such that for all positive integers $n$ and $m$ with
    \[m_0 \le m \le n^d\]
    there are distinct positive integers $x_1<\dots <x_s<n-1$ with
    \[\binom{x_1}{d}+\dots+\binom{x_s}{d}=m.\]
\end{prop}

Here, the exact shape of the upper bound on $m$ is not important, but note that we can certainly never represent numbers $m>s \cdot \binom{n}{d}$, so some restriction is clearly necessary.

We conclude this section by deducing the main result from these two propositions.

\begin{proof}[Proof of Theorem~\ref{thm:main}]
    We apply Proposition~\ref{prop:ant} with $d=k-1$ to find the existence of $s$ and $m_0$ such that all numbers $m$ with $m_0 \le m \le n^d$ have a representation of the form
    \[m=\binom{a_1}{k-1}+\dots+\binom{a_s}{k-1}\]
    with distinct $0<a_1<\dots<a_s<n-1$.
    In view of Proposition~\ref{prop:graph}, this implies that all numbers in the interval
    \[\left[(n-1)\binom{n-1}{k-1}-n^{k-1}, (n-1)\binom{n-1}{k-1}-m_0\right]\]
    appear as the Steiner--Wiener index of a graph. We conclude by noting that these intervals overlap for all sufficiently large $n$, by a simple computation.
\end{proof}

\subsection*{Acknowledgements}

We would like to thank Hua Wang and Xueliang Li for useful feedback on an earlier version.

\section{The graph-theoretic construction}

In this section, we prove Proposition~\ref{prop:graph}. Figure \ref{fig:construction} gives an illustration of our construction.

\begin{proof}[Proof of Proposition~\ref{prop:graph}]
    For given positive integers $a_1<\dots<a_r<n-1$, we construct the graph $G$ as follows: As the vertex set we choose $\{0,1,\dots,n-1\}$. For $0 \le i<j \le n-1$, we include an edge $(i,j)$ in $G$ iff $j \in \{a_1,\dots,a_r,n-1\}$. We now show that $G$ has the desired property.

    Indeed, let us note that since the vertex with label $n-1$ is connected to all other vertices, $G$ is a star graph. In particular, for any $S \subset V(G)$ of size $k$, we have $d(S) \in \{k-1,k\}$. To compute $\SW(S_n)-\SW(G)$, we therefore need to compute the number of such $S$ not containing the vertex $n-1$ which have $d(S)=k-1$ (those that contain the vertex $n-1$ are already taken into account in $\SW_k(S_n)$). But this clearly happens iff $\max(S) \in \{a_1,\dots,a_r\}$. However, there are exactly $\binom{a_i}{k-1}$ sets $S$ with $\max(S)=a_i$.
\end{proof}

\begin{remark*}
    Note that even though all our constructed graphs are star graphs of diameter $2$, one can adapt the construction to solve the inverse problem for graphs of any given diameter using a broom graph instead of a star as the initial template.
\end{remark*}

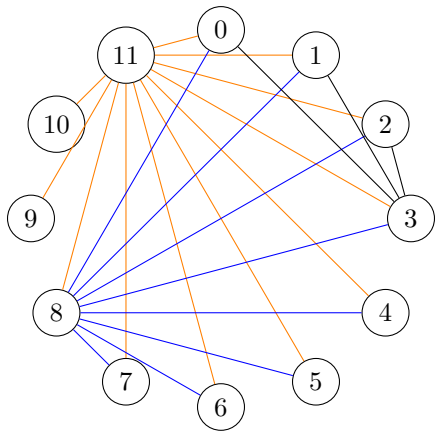
\begin{figure}[ht]
\begin{center}
\begin{tikzpicture}[scale=2.5]

\def\n{12}

\foreach \i in {0,...,11} {
    \node[circle, draw, fill=white] (N\i) 
        at ({90 - 360/\n * \i}:1) {\i};
}

\foreach \a in {0,...,10} {
     \draw[orange] (N\a) -- (N11);
 }

\foreach \a in {0,...,7} {
     \draw[blue] (N\a) -- (N8);
 }

\foreach \a in {0,...,2} {
     \draw[black] (N\a) -- (N3);
}

\end{tikzpicture}
\caption{The nested star construction for $n=12, r=2$ with $a_1=3, a_2=8$}\label{fig:construction}
\end{center}
\end{figure}

\section{Counting solutions to diophantine equations}

In this section, we prove Proposition~\ref{prop:ant}. The key technical input is the following asymptotic counting result for the number of representations of a number $m$ as a sum of binomial coefficients. To this end, for positive integers $m,n,d,s,\lambda_1,\dots,\lambda_s$, let $B=\lceil\frac{1}{100} m^{1/d}\rceil$ and consider the counting function
  \[N(m):=\#\{(x_1,\dots,x_s) \in \mathbb{N}^s: x_1,\dots,x_s \le B: \sum_{i=1}^s \lambda_i \binom{x_i}{d}=m\}.\]
\begin{prop}\label{prop:counting}
    For fixed $s$ and $d$ with $s \ge d! \cdot 1000^d$ we have
    \[N(m)=(1+o(1)) c_m \cdot m^{s/d-1}\]
    as $m \to \infty$, provided that more than $d \cdot d!$ of the variables $\lambda_i$ are equal to $1$.
    
    Here, $c_m=c_{m,\lambda_1,\dots,\lambda_s}$ is a constant depending on $m,\lambda_1,\dots,\lambda_s$, which satisfies $c_m \in [c,c']$ for some positive constants $c,c'$ depending only on $s$ and $d$. 
\end{prop}

Note that the bounds on $s$ and the $\lambda_i$ arise from a rather crude treatment of the local densities, but certainly some condition has to be imposed as a consequence of our various size restrictions. Since we are only aiming for an asymptotic result in our main theorem, we have made no attempt to optimize the constants.

\begin{proof}[Proof of Proposition~\ref{prop:ant}]
    Let $N^*(m)$ be defined similarly to $N(m)$ but with the $x_i$ required to be distinct. Note that the difference $N(m)-N^*(m)$ counts solutions where at least two of the variables are the same. But this amounts to finitely many equations of the same shape as before, with two $\lambda_i$ replaced by their sum, and with $s$ replaced by $s-1$. Applying this with $\lambda_1=\dots=\lambda_s=1$, by Proposition~\ref{prop:counting} we therefore have
    \[N(m)-N^*(m) \ll m^{(s-1)/d-1}=o(m^{s/d-1})\]
    and hence
    \[N^*(m)=(1+o(1)) c_m \cdot m^{s/d-1}\]
    as $m \to \infty$. In particular, if $m$ is sufficiently large, we will have $N^*(m)>0$, hence any $m \le n^d$ has a representation of the desired shape with $x_i \le \frac{1}{50}m^{1/d} \le n/50 < n-1$.
\end{proof}

Finally, we are left to establish our counting result, Proposition~\ref{prop:counting}. This is done by an application of the Hardy-Littlewood circle method, familiar from the resolution of Waring's problem. Most of the steps are relatively standard, so we will be brief and refer the reader to \cite{vaughan} for more details on the method. However, the fact that $\binom{x}{d}$ is an \emph{integer-valued} but not an \emph{integer coefficient} polynomial accounts for some complications in the treatment of the local densities.

\begin{proof}[Proof of Proposition~\ref{prop:counting}]
    For $\alpha \in \mathbb{R}$, let us consider the exponential sum
    \[S(\alpha)=\sum_{x=1}^B e\left(\alpha \binom{x}{d}\right)\]
    where $e(\alpha)=e^{2\pi i\alpha}$. We then have
    \[N(m)=\int_0^1 S(\lambda_1\alpha) \cdot \dotsc \cdot S(\lambda_s \alpha)e(-m\alpha) \ddd\alpha\]
    by the Fourier orthogonality relations. The next step is to split the interval of integration in the \emph{major arcs} $\mathfrak{M}$, consisting of small intervals around rational numbers with small denominators, and its complement $\mathfrak{m}$, the \emph{minor arcs}; the idea being that the contribution coming from $\mathfrak{M}$ should give the main term, while on $\mathfrak{m}$ the exponential sum $S(\alpha)$ is small due to oscillation, leading to a negligible contribution.

    More precisely, we let the major arcs $\mathfrak{M}$ to consist of numbers of the form $\alpha=\frac{a}{q}+\beta$ for some coprime integers $a$ and $q$ with $q \le B^{1/100}$ and a real number $\beta$ with $|\beta| \le B^{1/100-d}$.

    Weyl's inequality as in \cite[Lemma~2.4, Theorem~2.1]{vaughan} then shows that
    \[\sup_{\alpha \in \mathfrak{m}} |S(\lambda_i\alpha)| \ll B^{1-\frac{1}{100 \cdot 2^{d-1}}+\varepsilon}\]
    for any $\varepsilon>0$. Thus, for $s>100d \cdot 2^{d-1}$, we have
    \[N(m)=\int_{\mathfrak{M}} S(\lambda_1\alpha) \cdot \dotsc \cdot S(\lambda_s \alpha) e(-m\alpha) \ddd\alpha+o(m^{s/d-1}).\]
    As in \cite[Section~2.4]{vaughan}, we can proceed to evaluate the major arc contribution to obtain
    \[N(m)=\mathfrak{S}(m,B^{1/100}) \cdot \mathfrak{I}(m,B^{1/100}) \cdot B^{s-d}+o(B^{s-d})\]
    with the truncated \emph{singular series}
    \[\mathfrak{S}(m,Q):=\sum_{q \le Q} \sum_{(a;q)=1} \widetilde{S}(q,\lambda_1 a) \cdot \dotsc \cdot  \widetilde{S}(q,\lambda_s a) e\left(-\frac{am}{q}\right)\]
    defined using the exponential sum
    \[\widetilde{S}(q,a)=\frac{1}{q \cdot d!}\sum_{b \Mod{q \cdot d!}} e\left(\frac{a \cdot \binom{b}{d}}{q}\right)\]
    and the truncated \emph{singular integral}
    \[\mathfrak{I}(m,Q)=\int_{-Q}^Q v(\lambda_1\beta) \cdot \dotsc \cdot v(\lambda_s\beta) e(-m\beta/B^d) \ddd\beta\]
    defined using the exponential integral
    \[v(\beta)=\int_0^1 e\left(\beta \cdot t^d/d!\right) \ddd t.\]
    Note that the only essential change compared to the discussion in \cite[Section~2.4]{vaughan} is that the value of $\binom{x}{d}$ modulo $q$ depends on the value of $x$ modulo $q \cdot d!$, leading to the somewhat unusual definition of $\widetilde{S}(q,a)$. For the singular integral, we have already rescaled the variables leading to the factor $B^{s-d}$ and noted that the lower order terms in $\binom{t}{d}$ give a negligible contribution.

    In particular, up to the factor of $d!$, the singular integral is now the same as in Waring's problem and therefore the discussion in \cite[Section~2.5]{vaughan} shows that $\mathfrak{I}(m,Q)$ is uniformly bounded from above and below by positive constants, whenever $s>d! \cdot 1000^d$. Note that the rather large number of variables is required to ensure the existence of real solutions to our equations.

    It remains to show that $\mathfrak{S}(m,B^{1/100})$ is uniformly bounded from above and below by positive constants.
    The first step is to note that for $(a_1;q_1)=(a_2;q_2)=(q_1;q_2)=1$, as in \cite[Lemma~2.10]{vaughan} we still have
    \[\widetilde{S}(q_1q_2,a_1q_2+a_2q_1)=\widetilde{S}(q_1,a_1) \cdot \widetilde{S}(q_2,a_2).\]
    Indeed, the key observation here is that even though above we noted that the value of $\binom{x}{d}$ modulo $q$ a priori depends on the residue of $x$ modulo $q \cdot d!$, it really only depends on the residue of $x$ modulo $q \cdot (d!;q)$, allowing for the above decomposition.

    Completing the singular series as in \cite[Section~2.4]{vaughan} we therefore find that
    \[\mathfrak{S}(m,B^{1/100})=\mathfrak{S}(m)+o(1)\]
    with the factors of the Euler product $\mathfrak{S}(m)=\prod_p \chi_p$ defined via
    \[\chi_p=\sum_{k=0}^{\infty} \sum_{(a;p^k)=1} \widetilde{S}(p^k,\lambda_1 a) \cdot \dotsc \cdot  \widetilde{S}(p^k,\lambda_s a) e\left(-\frac{am}{p^k}\right).\]
    As in \cite[Theorem~2.4]{vaughan}, by Weyl's inequality the Euler product is absolutely convergent for $s>2^d$ and the contribution of $p>p_0(d)$ is between $1/2$ and $3/2$, so that it remains to consider the finitely many small primes $p \le p_0(d)$.

    For those, writing $v_p(d!)=t$, we still have
    \[\chi_p=\lim_{k \to \infty} \frac{p^k}{(p^{k+t})^s} \cdot M_m(p^k)\]
    as in \cite[Lemma~2.12]{vaughan} where $M_m(p^k)$ is the number of solutions of 
    \[\lambda_1\binom{x_1}{d}+\dots+\lambda_s\binom{x_s}{d} \equiv m \Mod{p^k}\]
    with $x_1,\dots,x_s$ running through residue classes modulo $p^{k+t}$.

    We claim that $M_m(p^{1+t})>0$, i.e. we can find a solution modulo $p^{1+t}$. Indeed, by setting some variables to zero, it suffices to consider the case $s=d \cdot d!$ where all $\lambda_i$ are equal to $1$. Writing $B_d$ for the set of residues of binomial coefficients $\binom{x}{d}$ modulo $p^{1+t}$ as $x$ runs through all integers, we want to prove that $s \cdot B_d=\mathbb{Z}/p^{1+t}\mathbb{Z}$.

    If $p>d$, note that $t=0$ so that $d!$ is coprime to $p$ and the size of $B_d$ is the same as the number of values the integer polynomial $x(x-1)\dots (x-d+1)$ takes modulo $p$. But since every value occurs at most $d$ times, we find that $|B_d| \ge \frac{p}{d}$, which by the Cauchy-Davenport theorem is more than enough to deduce that $s \cdot B_d=\mathbb{Z}/p\mathbb{Z}$.

    On the other hand, if $p<d$, we always have $\{0,1\} \subset B_d$ so that certainly $s \cdot B_d=\mathbb{Z}/p^{1+t}\mathbb{Z}$ as long as $s \ge p^{1+t}$ which is satisfied since $p^{1+t} \le d \cdot d!$.
    Note that the previous argument shows that when $s>d \cdot d!$ we can even ensure the existence of a solution with $x_1=0$.

    We can now use this solution to lift our solution modulo $p^{1+t}$ to a solution modulo $p^k$ for $k>t$: For any of the $p^{(k-t-1)(s-1)}$ values of $x_2,\dots,x_n$ modulo $p^{k+t}$ congruent to our initial solution modulo $p^{1+2t}$, we claim that we can find $x_1$ solving our congruence modulo $p^k$.

    Indeed, it suffices to show that the values of $\binom{x_1}{d}$ cover all residues $\equiv 0 \Mod{p^{1+t}}$ or equivalently, the values of $f(x_1)=x_1(x_1-1)\dots (x_1-d+1)$ cover all residues $\equiv 0 \Mod{p^{1+2t}}$. But since $v_p(f'(0)) \le t$, this is the consequence of an application of Hensel's lemma. We have thus established
    \[M_m(p^k) \ge p^{(k-t-1)(s-1)} M_m(p^{1+t}) \ge p^{(k-t-1)(s-1)}\]
    and hence
    \[\chi_p \ge p^{-ts-(t+1)(s-1)} \gg_{d,s} 1,\]
    finishing our proof.
\end{proof}

\bibliographystyle{plain}

\bibliography{bibliography}

\end{document}